\documentclass[runningheads]{llncs}
\usepackage{hyperref}
\usepackage{amsmath}
\usepackage{amssymb}
\usepackage{diagbox}
\usepackage{arydshln}
\usepackage{makecell}
\usepackage{enumitem}

\newcommand{\BP}{\mathbf{P}}

\newcommand{\M}{\mathfrak{M}}

\newcommand{\Val}{\text{\rm Val}}

\newcommand{\TVL}{\text{\sf 3VL}}
\newcommand{\TVML}{\text{\sf 3VML}}
\newcommand{\STVMLI}{\text{\sf S3VML--\uppercase\expandafter{\romannumeral1}}}
\newcommand{\STVMLII}{\text{\sf S3VML--\uppercase\expandafter{\romannumeral2}}}

\newcommand{\wK}{\text{\tt wK}}
\newcommand{\one}{\text{\tt\uppercase\expandafter{\romannumeral1}}}
\newcommand{\two}{\text{\tt\uppercase\expandafter{\romannumeral2}}}

\newcolumntype{L}{>{$}l<{$}}
\newcolumntype{R}{>{$}r<{$}}
\newcolumntype{C}{>{$}c<{$}}

\begin{document}

\title{\texorpdfstring{On Three-Valued Modal Logics:\\from a Four-Valued Perspective}{On Three-Valued Modal Logics: from a Four-Valued Perspective}}
\titlerunning{On Three-Valued Modal Logics: from a Four-Valued Perspective}

\author{Xinyu Wang \and Yang Song \and Satoshi Tojo}
\authorrunning{X. Wang et al.}

\institute{School of Information Science, Japan Advanced Institute of Science and Technology\\
Nomi City, Ishikawa Prefecture, 923--1211, Japan}

\maketitle

\begin{abstract}
This paper aims at providing a comprehensive solution to the archaic open problem: how to define semantics of three-valued modal logic with vivid intuitive picture, convincing philosophical justification as well as versatile practical usage. Based on an existing line of work concerned with investigating three-valued logic out of innovative angles of view, we adopt a detour approach to interpret three-valued logic from a four-valued perspective, which results in the invention of an universal and systematic methodology for developing, explaining as well as utilizing three-valued modal logic. We illustrate our method through two concrete cases, one deontic and another epistemic, for both of which a sound and strongly complete natural deduction proof system is also presented in detail. We perceive our three-valued modal logic as a lightweight candidate to merge deontic or epistemic notion into temporal logic, without heavier burden of multiple modalities.

\keywords{Three-valued logic \and Four-valued logic \and Modal logic \and Weak Kleene logic \and Kripke model \and Natural deduction.}
\end{abstract}

\section{Introduction}\label{sec.intro}

Classical propositional logic possesses only two distinct truth values, namely, a valuation function assigns either \textit{True} ($T$) or \textit{False} ($F$) to every basic propositional letter, and hence any logic formula is two-valued as well. In many cases, such kind of logic formalization well captures our na\"{i}ve perception that external facts are absolute, therefore, any statement must be either true or false \textit{objectively}. In other words, \textit{Tertium Non Datur} perpetually holds for each proposition, regardless of \textit{subjectively}, whether we know at present or even will possibly know the proposition's real truth value~\cite{Ditmarsch08}.

Nonetheless, perhaps since our material world never looks perfect, neither does two-valued logic. In fact, ever since intuitionistic logic's discovery, we have witnessed the spawning of a great family of varied non-classical logics~\cite{Fitting94}, which generally aim at compensating myriad aspects of classical logic's deficiencies. Among all sorts of non-classical logics that extend beyond classical propositional logic, three-valued logic ought to serve properly as a solid foundation for us to start working with, as it enlarges two-valued logic minimally in the sense that it simply adds one more extra acceptable truth value. Further, many-valued logics equipped with more than three possible truth values have been proposed as well, however it is not necessarily the more, the better. As a matter of fact, being essentially unlike many-valued logic, which may allow an arbitrary number of different truth values and so actually behaves much similar to universal algebra~\cite{Galatos07}, three-valued logic instead focuses on justifying the third truth value through specific philosophical intuitions. Thus this extra third truth value, depending on its interpretation in the particular scenario, is denoted as $U$ more or less frequently. For instance, $U$ may be admittedly interpreted as \textit{Uncertain}, \textit{Undefined} or \textit{Unknown}~\cite{Hallden49}, and indeed, any of the above meaning is commonly recognized in everyday natural language as another feasible state for a proposition, i.e., some sort of ``middle state'' or ``grey zone'' beside \textit{True} and \textit{False}. Therefore in this paper, we choose three-valued logic to form our work's fundamental basis, for it veritably bears the name ``logic'', which is orthodoxly intended for deductive reasoning rather than algebraic calculi~\cite{Dalen13}.

Another celebrated direction to expand classical logic is to enrich the language with modality, resulting in prosperity of propositional modal logic~\cite{Blackburn06}. Apparently then, here immediately arises such a natural problem: how to combine three-valued logic with modal logic into so-called ``three-valued modal logic''. As innocent the above idea sounds, the major hurdle to a satisfactory solution lies substantially not on technical complexity but instead on philosophical explication, namely, how we should define as well as justify semantics of modality $\Box$ under a three-valued situation. In fact, several primitive attempts to directly generalize the semantics were carried out over the decades~\cite{Segerberg68,Schotch78,Morikawa89,Correia02}, but unfortunately none of the work has been broadly acknowledged as convincing enough among most logicians, and therefore alternative philosophical comprehensions are still being come up with along recent years~\cite{Priest08,Beall16,Szmuc19}. On the other hand, frontal attack may not always be the nearest road to success. For the purpose of clearer philosophical elucidation, sometimes adopting a detour approach presents another practicable measure.

A series of works help to shed light on the above problem. It was first suggested by Suszko in~\cite{Suszko75} that there possibly exists profound association between $n$-valued logic and $m$-valued logic even if $n\neq m$. This idea later got developed further by R\'{e} and Szmuc in~\cite{Re19}. Successively, Song et al. in~\cite{Song19} advocated a novel interpretation for three-valued logic, through a detour of using four-valued logic as intermediate representation. Mathematically speaking, four-valued logic obviously possesses richer expressivity, but far more significantly, it has been found out that four-valued logic is able to naturally simulate many important kinds of three-valued logics, moreover, in an even clearer way that is much friendlier to intuitive understanding than the original three-valued semantics. Such natural simulation exactly provides us with a systematic method to work out semantics of three-valued modal logic. Our obtained three-valued modal logic presents a framework that closely resembles ordinary two-valued modal logic, hence making three-valued modal logic conceptually self-explanatory as well as practically expressive. Further, numerous mathematical properties of two-valued modal logic can also be automatically inherited as corresponding three-valued ones~\cite{Chagrov97}.

The rest of this paper is organized as the following: Section~\ref{sec.pre} introduces mathematical preliminaries, briefly reviewing Song et al.'s work~\cite{Song19} on how to interpret three-valued logic with four-valued auxiliary; Section~\ref{sec.caseI} and Section~\ref{sec.caseII} analyze two different cases of three-valued modal logics with concrete applications, one deontic and another epistemic, respectively, for both of which we provide a sound and strongly complete natural deduction system; Section~\ref{sec.con} concludes the whole paper and discusses over conceivable future research directions.

\section{\texorpdfstring{Three-Valued Propositional Logic \&\\Auxiliary Four-Valued Interpretation}{Three-Valued Propositional Logic and Auxiliary Four-Valued Interpretation}}\label{sec.pre}

Syntax of three-valued propositional logic is quite straightforward:

\begin{definition}[Language $\TVL$]
Given a non-empty countable set of propositional letters $\BP$, well formed formula $A$ in Language $\TVL$ is recursively defined as the following BNF, where $p\in\BP$:

\begin{align*}
    A::=p\mid\neg A\mid A\land A\mid A\lor A
\end{align*}
\end{definition}

\subsubsection*{Important Note} Familiar as the above syntax may seem, readers must pay attention to the fact that logical connectives $\neg$, $\land$ and $\lor$ are all fundamental symbols in Language $\TVL$, rather than abbreviations. Actually as three-valued logic, logical connectives naturally bear different semantics from classical two-valued logic. Hence, in order to avoid any further confusion, we will not introduce logical connective $\to$ in this paper. By the way, at the level of meta-language, namely English in which this paper is written, we still cling to two-valued logic.

Models of three-valued propositional logic are, of course, three-valued. As explained in Section~\ref{sec.intro}, we tentatively follow the routine to denote the third truth value as $U$, in addition to other two standard truth values $T$ and $F$:

\begin{definition}[Three-Valued Propositional Model]
A three-valued propositional model is a three-valued valuation function $V:\BP\to\{T,U,F\}$.
\end{definition}

As for semantics, in fact, several different versions of three-valued propositional logics exist, among which as an instance we would like to introduce weak Kleene logic here, one of the most famous and useful three-valued propositional logics~\cite{Ferguson15}:

\begin{definition}[Weak Kleene Valuation]\label{def.valwK}
Given a fixed three-valued propositional model $V$, for any $\TVL$-formula $A$, its weak Kleene valuation $\Val^V_\wK(A)$ is defined recursively as the following:

\begin{align*}
    \Val^V_\wK(p)=V(p) & & \begin{tabular}{R|CCCC}
        \text{\diagbox{$\Val^V_\wK(A)$}{$\Val^V_\wK(A\land B)$}{$\Val^V_\wK(B)$}} & \phantom{O} & T & U & F\\
        \hline
        & & & &\\
        T\phantom{T} & & T & U & F\\
        U\phantom{U} & & U & U & U\\
        F\phantom{F} & & F & U & F
    \end{tabular}\\
    \begin{tabular}{R|L}
        \Val^V_\wK(A) & \Val^V_\wK(\neg A)\\
        \hline
        &\\
        T\phantom{T} & \phantom{F}F\\
        U\phantom{U} & \phantom{U}U\\
        F\phantom{F} & \phantom{T}T
    \end{tabular} & & \begin{tabular}{R|CCCC}
        \text{\diagbox{$\Val^V_\wK(A)$}{$\Val^V_\wK(A\lor B)$}{$\Val^V_\wK(B)$}} & \phantom{O} & T & U & F\\
        \hline
        & & & &\\
        T\phantom{T} & & T & U & T\\
        U\phantom{U} & & U & U & U\\
        F\phantom{F} & & T & U & F
    \end{tabular}
\end{align*}
\end{definition}

\begin{definition}[Weak Kleene Semantics]\label{def.semwK}
For any three-valued propositional model $V$ and any $\TVL$-formula $A$, $V\vDash_\wK A$ iff $\Val^V_\wK(A)=T$.
\end{definition}

It may sound amazing at first, but an important fact about weak Kleene logic is that there exists no tautology, that is to say, no $\TVL$-formula's weak Kleene valuation always keeps to be $T$ in any three-valued propositional model. Thus, we should not be able to syntactically deduce any theorem of weak Kleene logic, either, and so weak completeness is trivial. Nonetheless, if given a non-empty set of premises, starting from which we may then be able to deduce some other conclusions, hence strong completeness is still meaningful. As a matter of fact, Petrukhin in~\cite{Petrukhin16} has recently proposed a sound and strongly complete natural deduction proof system for weak Kleene logic.

As expounded in Section~\ref{sec.intro}, Song et al. in~\cite{Song19} devises a novel methodology of interpreting three-valued propositional logic with the assistance of four-valued propositional logic. Here please allow us to present only a very brief digest of their work, which is of course not meant to be strictly formal by any means. Define a four-valued propositional model as a four-valued valuation function $V_4:\BP\to\{T_1,F_1\}\times\{T_2,F_2\}$, therefore simply assigning two independent two-valued truth values to every propositional letter (and also to every $\TVL$-formula). The core philosophical idea is that in the finest-grained view, everything is ultimately two-valued, for example, we can ably pick any one out of arbitrarily finite many possible values by just asking a series of \textit{yes/no} questions. Hence the pair of truth values $V_4(p)=(\Val^{V_4}_1(p),\Val^{V_4}_2(p))\in\{T_1,F_1\}\times\{T_2,F_2\}$ just represent two different \textit{yes/no} properties of the ``bundled'' propositional letter $p$, but when we zoom out to a courser-grained view, resolution decreases and $p$ blurs so as to look like one solitary three-valued propositional letter. Thus, the heart of the whole story hitherto settles on semantics of these two two-valued truth values, as well as a ``compression'' function $f_C:\{T_1,F_1\}\times\{T_2,F_2\}\to\{T,U,F\}$.

As for the case of weak Kleene logic, we let $\Val^{V_4}_1$ behave classically:

\begin{align}
    \label{for.begin}\Val^{V_4}_1(\neg A)=T_1 & \iff\Val^{V_4}_1(A)=F_1\\
    \Val^{V_4}_1(A\land B)=T_1 & \iff\Val^{V_4}_1(A)=T_1\text{ and }\Val^{V_4}_1(B)=T_1\\ \Val^{V_4}_1(A\lor B)=T_1 & \iff\Val^{V_4}_1(A)=T_1\text{ or }\Val^{V_4}_1(B)=T_1
\end{align}

We let $\Val^{V_4}_2$ behave \textit{False}-infectiously:

\begin{align}
    \Val^{V_4}_2(\neg A)=F_2 & \iff\Val^{V_4}_2(A)=F_2\\
    \Val^{V_4}_2(A\land B)=F_2 & \iff\Val^{V_4}_2(A)=F_2\text{ or }\Val^{V_4}_2(B)=F_2\\ \label{for.end}\Val^{V_4}_2(A\lor B)=F_2 & \iff\Val^{V_4}_2(A)=F_2\text{ or }\Val^{V_4}_2(B)=F_2
\end{align}

We let $f_C(T_1,T_2)=T$, $f_C(F_1,T_2)=F$, and $f_C(T_1,F_2)=f_C(F_1,F_2)=U$.

It can then be easily verified that the whole definition above surely conforms to weak Kleene valuation in Definition~\ref{def.valwK}. The advantage of this four-valued interpretation is straightforward: with its assistance, we can easily expand three-valued propositional logic onto three-valued modal logic, with ample confidence to philosophically justify our choice of definition for modality $\Box$'s three-valued semantics, since we already have a good intuition about how $\Box$ may act upon a two-valued truth value. Therefore, in the rest of this paper, while sticking to weak Kleene logic as the basic semantics for the propositional fragment of three-valued modal logic, we focus on designing plausible three-valued semantics for modality $\Box$ with the guidance of the above four-valued interpretation.

\section{Three-Valued Modal Logic: Case \uppercase\expandafter{\romannumeral1} (Deontic)}\label{sec.caseI}

To start with, we define syntax and model of three-valued modal logic:

\begin{definition}[Language $\TVML$]
Given a non-empty countable set of propositional letters $\BP$, well formed formula $A$ in Language $\TVML$ is recursively defined as the following BNF, where $p\in\BP$:

\begin{align*}
    A::=p\mid\neg A\mid A\land A\mid A\lor A\mid\Box A
\end{align*}
\end{definition}

\begin{definition}[Three-Valued Kripke Model]
A three-valued Kripke model $\M$ is a triple $(S,R,V)$ where:

\begin{itemize}
    \item $S$ is a non-empty set of possible worlds.
    \item $R\subseteq S\times S$ is a binary relation on $S$.
    \item $V:S\times\BP\to\{T,U,F\}$ is a three-valued valuation function.
\end{itemize}

A pointed model $\M,s$ is a model $\M$ with a possible world $s\in S$.
\end{definition}

Now we face the central problem: to define semantics of three-valued modal logic. As analyzed in Section~\ref{sec.pre}, we take a detour approach by first giving a specific interpretation for the auxiliary four-valued logic. Concerning Case \uppercase\expandafter{\romannumeral1} within this section, we choose to apply a deontological connotation: suppose $A$ is any $\TVL$-formula, for its first two-valued truth value namely $\Val^{V_4}_1(A)$, $T_1$ means the agent is obligated to do $A$, and so $F_1$ means the agent does not have to do $A$; for its second two-valued truth value namely $\Val^{V_4}_2(A)$, $T_2$ means the agent is allowed to do $A$, and so $F_2$ means the agent is forbidden to do $A$. Thus $f_C(T_1,T_2)=T$ means the agent must do $A$, $f_C(F_1,T_2)=F$ means the agent can either do $A$ or not do $A$, and $f_C(T_1,F_2)=f_C(F_1,F_2)=U$ means the agent must not do $A$ since ethically speaking, an immoral deed is afterall immoral even if it is also an obligation, for example, a soldier kills an enemy on the battlefield. Readers can intuitively reason that the above deontic interpretation actually fits quite properly into semantics defined in Equivalences (\ref{for.begin})--(\ref{for.end}). Further we designate a temporal interpretation to modality $\Box$ in Language $\TVML$, then semantics of $\Box$ can be assigned as the following:

\begin{enumerate}
    \item At any possible world, $\Val^{V_4}_1(\Box A)=T_1$ iff on all successors $\Val^{V_4}_1(A)=T_1$, because that the agent must keep doing $A$ all the time in the future is the same as that at any time in the future the agent must be doing $A$.
    \item At any possible world, $\Val^{V_4}_2(\Box A)=T_2$ iff on all successors $\Val^{V_4}_2(A)=T_2$, because that the agent is allowed to keep doing $A$ all the time in the future is the same as that at any time in the future the agent is allowed to do $A$.
\end{enumerate}

The above four-valued semantics can be precisely mapped down to three-valued semantics as the following:

\begin{definition}[Semantics \uppercase\expandafter{\romannumeral1}]\label{def.semI}
Given a fixed three-valued Kripke model $\M$, for any $\TVML$-formula $A$, definition of its valuation for the propositional fragment remains the same as Definition~\ref{def.valwK}, just adding the current possible world $s\in S$ so as to obtain $\Val^\M_\one(s,A)$, while for modality $\Box$:

\begin{align*}
    \Val^\M_\one(s,\Box A)=\left\{\begin{array}{lll}
        T, & & \text{if }\forall sRt,\Val^\M_\one(t,A)=T\\
        U, & & \text{if }\exists sRt,\Val^\M_\one(t,A)=U\\
        F, & & \text{otherwise}
    \end{array}\right.
\end{align*}

Similar to Definition~\ref{def.semwK}, for any three-valued pointed Kripke model $\M,s$ and any $\TVML$-formula $A$, $\M,s\vDash_\one A$ iff $\Val^\M_\one(s,A)=T$.
\end{definition}

Although as explained, we tend to interpret $\Box$ as temporal modality, for the time being let us not consider any particular restrictions over frames and thus firstly work out the most general proof system for normal three-valued modal logic:

\begin{definition}[Proof System $\STVMLI$]
The following natural deduction rules consist of Proof System $\STVMLI$, where for any $\TVML$-formula $A$, $[A]$ means that $A$ gets discharged from the set of premises:

\begin{table}
\centering
\large
\renewcommand{\arraystretch}{1.8}
\setlength{\tabcolsep}{9pt}
\caption{Proof System $\STVMLI$}
\begin{tabular}{LC:LC:LC}
    (EFQ) & \frac{A\quad\neg A}{B} & (\neg\neg I) & \frac{A}{\neg\neg A} & (\neg\neg E) & \frac{\neg\neg A}{A}\\
    (\lor I_1) & \frac{\neg A\land B}{A\lor B} & (\lor I_2) & \frac{A\land\neg B}{A\lor B} & (\lor I_3) & \frac{A\land B}{A\lor B}\\
    (\lor E) & \multicolumn{5}{L}{\frac{\small\renewcommand{\arraystretch}{0.9}\setlength{\arraycolsep}{9pt}\begin{array}{cccc}& [A\land B] & [A\land\neg B] & [\neg A\land B]\\A\lor B & C & C & C\end{array}}{\small C}}\\
    (\land I) & \frac{A\quad B}{A\land B} & (\land E_1) & \frac{A\land B}{A} & (\land E_2) & \frac{A\land B}{B}\\
    (\Box I) & \multicolumn{5}{l}{\normalsize from $\Gamma\vdash_\one A$, deduce $\{\Box B\mid B\in\Gamma\}\vdash_\one\Box A$}\\
    (\neg\lor I) & \frac{\neg A\land\neg B}{\neg(A\lor B)} & (\neg\lor E) & \frac{\neg(A\lor B)}{\neg A\land\neg B} & (\neg\land I) & \frac{\neg A\lor\neg B}{\neg(A\land B)}\\
    (\neg\land E) & \frac{\neg(A\land B)}{\neg A\lor\neg B} & (\neg\Box I) & \frac{\Box(A\lor\neg A)}{\Box A\lor\neg\Box A} & (\neg\Box E) & \frac{\Box A\lor\neg\Box A}{\Box(A\lor\neg A)}
\end{tabular}
\end{table}
\end{definition}

\begin{theorem}[Soundness]\label{th.soundI}
For any set of $\TVML$-formulae $\Gamma$ and any $\TVML$-formula $A$, $\Gamma\vdash_\one A\implies\Gamma\vDash_\one A$.
\end{theorem}

\begin{proof}
Soundness can be easily verified.
\end{proof}

\begin{definition}[Consistency]
A set of $\TVML$-formulae $\Gamma$ is inconsistent in Proof System $\STVMLI$, iff for any $\TVML$-formula $A$, $\Gamma\vdash_\one A$. $\Gamma$ is consistent iff it is not inconsistent.
\end{definition}

\begin{definition}[Maximal Consistent Set]
A set of $\TVML$-formulae $\Gamma$ is a maximal consistent set (abbreviated as MCS) in Proof System $\STVMLI$, iff all the following conditions hold:

\begin{itemize}
    \item $\Gamma$ is consistent.
    \item For any $\TVML$-formula $A$, if $\Gamma\vdash_\one A$ then $A\in\Gamma$.
    \item For any $\TVML$-formula $A\lor B\in\Gamma$, $\{A\land B,A\land\neg B,\neg A\land B\}\cap\Gamma\neq\emptyset$.
\end{itemize}
\end{definition}

\begin{lemma}[Lindenbaum Lemma]
For any set of $\TVML$-formulae $\Gamma$ and any $\TVML$-formula $A$, if $\Gamma\nvdash_\one A$, then there exists an MCS $\Lambda\supseteq\Gamma$ such that $\Lambda\nvdash_\one A$.
\end{lemma}

\begin{proof}
We build $\Lambda$ by induction. As Language $\TVML$ is countable, firstly fix an arbitrary well order $<$ isomorphic to $\omega$ over the set of all $\TVML$-formulae. Let $\Lambda_0=\{B\mid\Gamma\vdash_\one B\}$, it is easy to see that for any $\TVML$-formula $B$, $\Lambda_0\vdash_\one B\iff\Gamma\vdash_\one B$, so $\Lambda_0\nvdash_\one A$, and for any $\TVML$-formula $B$, if $\Lambda_0\vdash_\one B$ then $B\in\Lambda_0$. Now suppose $\Lambda_n$ has been constructed where $n\in\omega$, such that $\Lambda_n\nvdash_\one A$, and that for any $\TVML$-formula $B$, if $\Lambda_n\vdash_\one B$ then $B\in\Lambda_n$. We then construct $\Lambda_{n+1}$ as the following:

\begin{itemize}
    \item If for any $\TVML$-formula $B\lor C\in\Lambda_n$, $\{B\land C,B\land\neg C,\neg B\land C\}\cap\Lambda_n\neq\emptyset$, then let $\Lambda_{n+1}=\Lambda_n$.
    \item Otherwise, suppose $B\lor C\in\Lambda_n$ is the least $\TVML$-formula by the prespecified well order $<$ that does not satisfy the above condition. We claim that, at least one of the following holds: $\Lambda_n\cup\{B\land C\}\nvdash_\one A$, $\Lambda_n\cup\{B\land\neg C\}\nvdash_\one A$, $\Lambda_n\cup\{\neg B\land C\}\nvdash_\one A$. For suppose not, then by Rule $\lor E$ we can easily deduce $\Lambda_n\vdash_\one A$, contradicting the induction hypothesis. Since three cases are just similar, as an instance suppose $\Lambda_n\cup\{B\land C\}\nvdash_\one A$, then let $\Lambda_{n+1}=\{D\mid\Lambda_n\cup\{B\land C\}\vdash_\one D\}$. $\Lambda_{n+1}\nvdash_\one A$, and for any $\TVML$-formula $D$, if $\Lambda_{n+1}\vdash_\one D$ then $D\in\Lambda_{n+1}$.
\end{itemize}

Finally, let $\Lambda=\bigcup\limits_{n\in\omega}\Lambda_n$ be the union of this monotone set sequence. It is quite easy to reason that $\Gamma\subseteq\Lambda$, $\Lambda\nvdash_\one A$, so $\Lambda$ is consistent, for any $\TVML$-formula $B$, if $\Lambda\vdash_\one B$ then $B\in\Lambda$, and for any $\TVML$-formula $B\lor C\in\Lambda$, $\{B\land C,B\land\neg C,\neg B\land C\}\cap\Lambda\neq\emptyset$. Therefore $\Lambda$ is the desired MCS.
\end{proof}

\begin{theorem}[Strong Completeness]\label{th.compI}
For any set of $\TVML$-formulae $\Gamma$ and any $\TVML$-formula $A$, $\Gamma\vDash_\one A\implies\Gamma\vdash_\one A$.
\end{theorem}

\begin{proof}
Equivalently, we prove that $\Gamma\nvdash_\one A\implies\Gamma\nvDash_\one A$. Suppose $\Gamma\nvdash_\one A$, then by Lindenbaum Lemma there exists an MCS $\Lambda\supseteq\Gamma$ such that $\Lambda\nvdash_\one A$. Build the canonical model $\M$ as $(S,R,V)$ where:

\begin{itemize}
    \item $S$ is the set of all MCSs.
    \item $R\subseteq S\times S$ is a binary relation on $S$ such that for any $\Delta,\Theta\in S$, $\Delta R\Theta$ iff $\forall\Box B\in\Delta$, $B\in\Theta$.
    \item $V:S\times\BP\to\{T,U,F\}$ is a three-valued valuation function such that for any $\Delta\in S$ and any $p\in\BP$, if $p\in\Delta$, then $V(\Delta,p)=T$; if $\neg p\in\Delta$, then $V(\Delta,p)=F$; otherwise, $V(\Delta,p)=U$.
\end{itemize}

We then prove Truth Lemma by structural induction: for any MCS $\Delta\in S$ and any $\TVML$-formula $B$, $\M,\Delta\vDash_\one B$ iff $B\in\Delta$, and $\M,\Delta\vDash_\one\neg B$ iff $\neg B\in\Delta$.

For the basic case, for any MCS $\Delta\in S$ and any propositional letter $p\in\BP$, by definition of valuation function $V$ it is easy to see that $\M,\Delta\vDash_\one p$ iff $p\in\Delta$, and $\M,\Delta\vDash_\one\neg p$ iff $\neg p\in\Delta$.

For the inductive step, due to limited space here we only demonstrate the crucial case concerning modality $\Box$. As induction hypothesis, suppose that for any MCS $\Delta\in S$ and for certain fixed $\TVML$-formula $B$, $\M,\Delta\vDash_\one B$ iff $B\in\Delta$, and $\M,\Delta\vDash_\one\neg B$ iff $\neg B\in\Delta$.

\begin{itemize}
    \item If $\Box B\in\Delta$, then for any MCS $\Theta\in S$ such that $\Delta R\Theta$, $B\in\Theta$. By induction hypothesis $\M,\Theta\vDash_\one B$, so $\M,\Delta\vDash_\one\Box B$.
    \item If $\neg\Box B\in\Delta$, on the one hand, by Rules $\land I$ and $\lor I_1$ we have $\Box B\lor\neg\Box B\in\Delta$, and so by Rule $\neg\Box E$ we have $\Box(B\lor\neg B)\in\Delta$, thus for any MCS $\Theta\in S$ such that $\Delta R\Theta$, $B\lor\neg B\in\Theta$, and because $\Theta$ is an MCS, $B\in\Theta$ or $\neg B\in\Theta$, so by induction hypothesis we have $\M,\Theta\vDash_\one B$ or $\M,\Theta\vDash_\one\neg B$. On the other hand, as $\Box B\notin\Delta$, by Rule $\Box I$ we have $\{C\mid\Box C\in\Delta\}\nvdash_\one B$, and by Lindenbaum Lemma there exists an MCS $\Theta\in S$ such that $\{C\mid\Box C\in\Delta\}\subseteq\Theta$ and that $\Theta\nvdash_\one B$, so $\Delta R\Theta$ and $B\notin\Theta$, by induction hypothesis $\M,\Theta\nvDash_\one B$, thus $\M,\Theta\vDash_\one\neg B$. Hence $\M,\Delta\vDash_\one\neg\Box B$.
    \item If $\Box B\notin\Delta$ and $\neg\Box B\notin\Delta$, also because $\Delta$ is an MCS, by Rule $\neg\Box I$ we can reason that $\Box(B\lor\neg B)\notin\Delta$, similarly as above there exists an MCS $\Theta\in S$ such that $\Delta R\Theta$ and that $B\lor\neg B\notin\Theta$, so $B\notin\Theta$ and $\neg B\notin\Theta$, by induction hypothesis $\M,\Theta\nvDash_\one B$ and $\M,\Theta\nvDash_\one\neg B$, so $\M,\Delta\nvDash_\one\Box B$ and $\M,\Delta\nvDash_\one\neg\Box B$.
\end{itemize}

In all, we obtain that $\M,\Delta\vDash_\one\Box B$ iff $\Box B\in\Delta$, and $\M,\Delta\vDash_\one\neg\Box B$ iff $\neg\Box B\in\Delta$.

Therefore since $\Lambda$ is an MCS and $\Lambda\nvdash_\one A$, $A\notin\Lambda$, so by Truth Lemma $\M,\Lambda\nvDash_\one A$, but $\Gamma\subseteq\Lambda$ so by Truth Lemma $\M,\Lambda\vDash_\one\Gamma$, hence $\Gamma\nvDash_\one A$.
\end{proof}

Finally we conclude:

\begin{theorem}[Soundness and Strong Completeness]
Under Semantics \text{\rm\uppercase\expandafter{\romannumeral1}} in Definition~\ref{def.semI}, Proof System $\STVMLI$ is sound and strongly complete with respect to the class of all three-valued Kripke frames.
\end{theorem}

\begin{proof}
By Theorem~\ref{th.soundI} and Theorem~\ref{th.compI}.
\end{proof}

\begin{remark}
Considering that $\Box$ is interpreted as temporal modality, in fact with respect to the class of three-valued S4 Kripke frames, through very similar reasoning we can establish a sound and strongly complete proof system by replacing Rule $\Box I$ in Proof System $\STVMLI$ with the following Rule $\Box I'$, and also adding the following new rule $\Box E$:

\begin{table}
\centering
\large
\renewcommand{\arraystretch}{0.9}
\setlength{\tabcolsep}{9pt}
\begin{tabular}{lc:lc}
    ($\Box I'$) & \makecell[l]{\normalsize from $\{\Box B\mid B\in\Gamma\}\vdash_\one A$,\\\normalsize deduce$\{\Box B\mid B\in\Gamma\}\vdash_\one\Box A$} & ($\Box E$) & $\frac{\Box A}{A}$
\end{tabular}
\end{table}
\end{remark}

\section{Three-Valued Modal Logic: Case \uppercase\expandafter{\romannumeral2} (Epistemic)}\label{sec.caseII}

As for Case \uppercase\expandafter{\romannumeral2} within this section, through the similar detour approach, we start with providing a cognitive interpretation for the auxiliary four-valued logic: suppose $A$ is any $\TVML$-formula, for its first two-valued truth value namely $\Val^{V_4}_1(A)$, $T_1$ means objectively $A$ is true and so $F_1$ means $A$ is false, just as classical two-valued logic; for its second two-valued truth value namely $\Val^{V_4}_2(A)$, $T_2$ means the agent understands $A$, and so $F_2$ means the agent does not understand $A$. Thus $f_C(T_1,T_2)=T$ means the agent understands $A$ is true, $f_C(F_1,T_2)=F$ means the agent understands $A$ is false, and $f_C(T_1,F_2)=f_C(F_1,F_2)=U$ means the agent does not understand $A$ since under such a circumstance, it is sheer nonsense for the agent to talk about truth value of some statement that he does not even understand at all. Readers can intuitively reason that the above epistemic interpretation actually fits quite properly into semantics defined in Equivalences (\ref{for.begin})--(\ref{for.end}). Further we designate a temporal interpretation to modality $\Box$ in Language $\TVML$, then semantics of $\Box$ can be assigned as the following:

\begin{enumerate}
    \item At any possible world, $\Val^{V_4}_1(\Box A)=T_1$ iff on all successors $\Val^{V_4}_1(A)=T_1$, the same as classical modal logic.
    \item At any possible world, $\Val^{V_4}_2(\Box A)=T_2$ iff on the very same possible world $\Val^{V_4}_2(A)=T_2$, because understanding the meaning of a sentence depends solely on status quo, regardless whether the sentence itself talks about past, present or future.
\end{enumerate}

Moreover, it can be reasonably assumed that the agent never forgets his knowledge, so that once he understands the meaning of a propositional letter, he will then always understand it in the future~\cite{Wang13}. Namely, the following restriction should be put onto the Kripke model:

\begin{enumerate}[start=3]
    \item For any propositional letter $p\in\BP$, at any possible world if $\Val^{V_4}_2(p)=T_2$, then on all successors $\Val^{V_4}_2(p)=T_2$.
\end{enumerate}

This restriction can be precisely mapped down to restriction on three-valued Kripke model as the following:

\begin{definition}[Three-Valued Kripke Model--\uppercase\expandafter{\romannumeral2}]
A three-valued Kripke model--\text{\rm\uppercase\expandafter{\romannumeral2}} is a three-value Kripke model such that for any $s\in S$ and any $p\in\BP$, $V(s,p)=U\implies\forall tRs,V(t,p)=U$.
\end{definition}

And within this restricted class of three-valued Kripke models--\uppercase\expandafter{\romannumeral2}, the above four-valued semantics can be precisely mapped down to three-valued semantics as the following:

\begin{definition}[Semantics \uppercase\expandafter{\romannumeral2}]\label{def.semII}
Given a fixed three-valued Kripke model--\text{\rm\uppercase\expandafter{\romannumeral2}} $\M_\two$, for any $\TVML$-formula $A$, definition of its valuation for the propositional fragment remains the same as Definition~\ref{def.valwK}, just adding the current possible world $s\in S$ so as to obtain $\Val^{\M_\two}_\two(s,A)$, while for modality $\Box$:

\begin{align*}
    \Val^{\M_\two}_\two(s,\Box A)=\left\{\begin{array}{lll}
        T, & & \text{if }\Val^{\M_\two}_\two(s,A)\neq U\text{ and }\forall sRt,\Val^{\M_\two}_\two(t,A)=T\\
        U, & & \text{if }\Val^{\M_\two}_\two(s,A)=U\\
        F, & & \text{otherwise}
    \end{array}\right.
\end{align*}

Similar to Definition~\ref{def.semwK}, for any three-valued pointed Kripke model--\text{\rm\uppercase\expandafter{\romannumeral2}} $\M_\two,s$ and any $\TVML$-formula $A$, $\M_\two,s\vDash_\two A$ iff $\Val^{\M_\two}_\two(s,A)=T$.
\end{definition}

Although as explained, we tend to interpret $\Box$ as temporal modality, for the time being let us not consider any particular restrictions over frames and thus firstly work out the most general proof system for normal three-valued modal logic:

\begin{definition}[Proof System $\STVMLII$]
The following natural deduction rules consist of Proof System $\STVMLII$, where for any $\TVML$-formula $A$, $[A]$ means that $A$ gets discharged from the set of premises:

\begin{table}
\centering
\large
\renewcommand{\arraystretch}{1.8}
\setlength{\tabcolsep}{9pt}
\caption{Proof System $\STVMLII$}
\begin{tabular}{lc:lc:lc}
    ($EFQ$) & $\frac{A\quad\neg A}{B}$ & ($\neg\neg I$) & $\frac{A}{\neg\neg A}$ & ($\neg\neg E$) & $\frac{\neg\neg A}{A}$\\
    ($\lor I_1$) & $\frac{\neg A\land B}{A\lor B}$ & ($\lor I_2$) & $\frac{A\land\neg B}{A\lor B}$ & ($\lor I_3$) & $\frac{A\land B}{A\lor B}$\\
    ($\lor E$) & \multicolumn{5}{l}{$\frac{\small\renewcommand{\arraystretch}{0.9}\setlength{\arraycolsep}{9pt}\begin{array}{cccc}& [A\land B] & [A\land\neg B] & [\neg A\land B]\\A\lor B & C & C & C\end{array}}{\small C}$}\\
    ($\land I$) & $\frac{A\quad B}{A\land B}$ & ($\land E_1$) & $\frac{A\land B}{A}$ & ($\land E_2$) & $\frac{A\land B}{B}$\\
    ($\Box I_1$) & \multicolumn{3}{l:}{\makecell[l]{\small from $\Gamma\vdash_\two A$, deduce\\\small$\{A\lor\neg A\}\cup\{\Box B\mid B\in\Gamma\}\vdash_\two\Box A$}} & ($\Box I_2$) & $\frac{A\lor\neg A}{\Box(A\lor\neg A)}$\\
    ($\neg\lor I$) & $\frac{\neg A\land\neg B}{\neg(A\lor B)}$ & ($\neg\lor E$) & $\frac{\neg(A\lor B)}{\neg A\land\neg B}$ & ($\neg\land I$) & $\frac{\neg A\lor\neg B}{\neg(A\land B)}$\\
    ($\neg\land E$) & $\frac{\neg(A\land B)}{\neg A\lor\neg B}$ & ($\neg\Box I$) & $\frac{A\lor\neg A}{\Box A\lor\neg\Box A}$ & ($\neg\Box E$) & $\frac{\Box A\lor\neg\Box A}{A\lor\neg A}$
\end{tabular}
\end{table}
\end{definition}

\begin{theorem}[Soundness]\label{th.soundII}
For any set of $\TVML$-formulae $\Gamma$ and any $\TVML$-formula $A$, $\Gamma\vdash_\two A\implies\Gamma\vDash_\two A$.
\end{theorem}

\begin{proof}
Soundness can be easily verified.
\end{proof}

Consistency and maximal consistent set are defined in the same way as Case \uppercase\expandafter{\romannumeral1} in Section~\ref{sec.caseI}, so can Lindenbaum Lemma be proved as well. All we need to do is only modifying part of proof for strong completeness:

\begin{theorem}[Strong Completeness]\label{th.compII}
For any set of $\TVML$-formulae $\Gamma$ and any $\TVML$-formula $A$, $\Gamma\vDash_\two A\implies\Gamma\vdash_\two A$.
\end{theorem}

\begin{proof}
Quite similar to proof of Theorem~\ref{th.compI}, we only need to revise two parts.

Firstly, we have to verify that the canonical model, which we now denote as $\M_\two$, is indeed a three-valued Kripke model--\uppercase\expandafter{\romannumeral2}. For any $\Delta\in S$ and any $p\in\BP$ so that $V(\Delta,p)=U$, to a contradiction suppose there exists $\Theta\in S$ such that $\Theta R\Delta$ and that $V(\Theta,p)\neq U$, then $V(\Theta,p)=T$ or $V(\Theta,p)=F$. Therefore by definition of three-valued valuation function $V$, we have $p\in\Theta$ or $\neg p\in\Theta$ so $p\lor\neg p\in\Theta$, by Rule $\Box I_2$ we have $\Box(p\lor\neg p)\in\Theta$, hence $p\lor\neg p\in\Delta$, and because $\Delta$ is an MCS, $p\in\Delta$ or $\neg p\in\Delta$, contradicting that $V(\Delta,p)=U$.

Secondly, we modify proof in the inductive step concerning modality $\Box$ as:

\begin{itemize}
    \item If $\Box B\in\Delta$, on the one hand, by Rules $\neg\neg I$, $\land I$ and $\lor I_2$ we have $\Box B\lor\neg\Box B\in\Delta$, and so by Rule $\neg\Box E$ we have $B\lor\neg B\in\Delta$, because $\Delta$ is an MCS, $B\in\Delta$ or $\neg B\in\Delta$, by induction hypothesis $\M_\two,\Delta\vDash_\two B$ or $\M_\two,\Delta\vDash_\two\neg B$. On the other hand, for any MCS $\Theta\in S$ such that $\Delta R\Theta$, $B\in\Theta$, by induction hypothesis $\M_\two,\Theta\vDash_\two B$. Hence $\M_\two,\Delta\vDash_\two\Box B$.
    \item If $\neg\Box B\in\Delta$, on the one hand, by Rules $\land I$ and $\lor I_1$ we have $\Box B\lor\neg\Box B\in\Delta$, and so by Rule $\neg\Box E$ we have $B\lor\neg B\in\Delta$, because $\Delta$ is an MCS, $B\in\Delta$ or $\neg B\in\Delta$, by induction hypothesis $\M_\two,\Delta\vDash_\two B$ or $\M_\two,\Delta\vDash_\two\neg B$. On the other hand, as $B\lor\neg B\in\Delta$ and $\Box B\notin\Delta$, by Rule $\Box I_1$ we have $\{C\mid\Box C\in\Delta\}\nvdash_\two B$, and by Lindenbaum Lemma there exists an MCS $\Theta\in S$ such that $\{C\mid\Box C\in\Delta\}\subseteq\Theta$ and that $\Theta\nvdash_\two B$, so $\Delta R\Theta$ and $B\notin\Theta$, by induction hypothesis $\M_\two,\Theta\nvDash_\two B$. Hence $\M_\two,\Delta\vDash_\two\neg\Box B$.
    \item If $\Box B\notin\Delta$ and $\neg\Box B\notin\Delta$, also because $\Delta$ is an MCS, by Rule $\neg\Box I$ we reason $B\lor\neg B\notin\Delta$, so $B\notin\Delta$ and $\neg B\notin\Delta$, by induction hypothesis $\M_\two,\Delta\nvDash_\two B$ and $\M_\two,\Delta\nvDash_\two\neg B$, so $\M_\two,\Delta\nvDash_\two\Box B$ and $\M_\two,\Delta\nvDash_\two\neg\Box B$.
\end{itemize}
\end{proof}

Finally we conclude:

\begin{theorem}[Soundness and Strong Completeness]
Under Semantics \text{\rm\uppercase\expandafter{\romannumeral2}} in Definition~\ref{def.semII}, Proof System $\STVMLII$ is sound and strongly complete with respect to the class of all three-valued Kripke models--\text{\rm\uppercase\expandafter{\romannumeral2}}.
\end{theorem}

\begin{proof}
By Theorem~\ref{th.soundII} and Theorem~\ref{th.compII}.
\end{proof}

\begin{remark}
Considering that $\Box$ is interpreted as temporal modality, in fact with respect to the class of three-valued S4 Kripke models--\uppercase\expandafter{\romannumeral2}, through very similar reasoning we can establish a sound and strongly complete proof system by replacing Rule $\Box I_1$ in Proof System $\STVMLII$ with the following Rule $\Box I_1'$, and also adding the following new rule $\Box E$:

\begin{table}
\centering
\large
\renewcommand{\arraystretch}{0.9}
\setlength{\tabcolsep}{9pt}
\begin{tabular}{lc:lc}
    ($\Box I_1'$) & \makecell[l]{\normalsize from $\{\Box B\mid B\in\Gamma\}\vdash_\two A$,\\\normalsize deduce $\{\Box B\mid B\in\Gamma\}\vdash_\two\Box A$} & ($\Box E$) & $\frac{\Box A}{A}$
\end{tabular}
\end{table}
\end{remark}

\section{Conclusions and Future Work}\label{sec.con}

In this paper, we take an uncommon detour approach of interpreting three-valued weak Kleene logic by auxiliary four-valued logic, so as to obtain a deeper and clearer philosophical insight, which then guides us to evolve three-valued propositional logic into three-valued modal logic in a systematical way and spontaneously generates very natural three-valued semantics suitable for modality $\Box$. To demonstrate our method, two practical example cases are presented and analyzed in detail, with sound and strongly complete natural deduction proof systems. One case is deontic and another one is epistemic, both of which are quite interesting and popular topics in study of modal logic as well as philosophy, and our technique of three-valued modal logic provides a clean and elegant way to combine deontic or epistemic notion into temporal logic, without too much complexity to make use of multiple modalities.

Nevertheless, there still remains much work to be done in the future for further expatiation and generalization. For example, parallel to classical two-valued modal logic, concepts such as bisimulation, definability, characterization theorem, finite model property and computational complexity might apply to three-valued modal logic as well~\cite{Blackburn01}. Moreover, it also looks quite promising to wield similar measure on many-valued modal logic.

\section*{Acknowledgement}

The first author would like to thank Hajime Ishihara for useful advice and help on issues related to ituitionistic logic. The second and third authors would like to thank Hitoshi Omori, with whom prolonged philosophical discussion over many inspiring topics relevant to this paper has been carried out.

\bibliographystyle{splncs04}
\bibliography{main}

\end{document}